\documentclass[12pt,a4paper]{amsart}
\usepackage{algorithm}
\usepackage{algorithmic}
\usepackage{amssymb}
\usepackage{eucal}
\usepackage{graphicx}
\usepackage{amsmath}
\usepackage{amscd}
\usepackage[all]{xy}           
\usepackage{tikz}
\usepackage{tikz-qtree}
\usepackage{amsfonts,latexsym}
\usepackage{xspace}
\usepackage[T1]{fontenc}
\usepackage{tikz-cd}
\usepackage{txfonts}

\usepackage{epsfig}
\usepackage{float}
\usepackage{color}
\usepackage{fancybox}
\usepackage{colordvi}
\usepackage{multicol}
\usepackage{colordvi}
\ifpdf
\usepackage[colorlinks,final,backref=page,hyperindex]{hyperref}
\else
\usepackage[colorlinks,final,backref=page,hyperindex]{hyperref}
\fi
\usepackage[active]{srcltx} 
\usepackage{mathrsfs} 
\usepackage{diagbox}

\newcommand{\nc}{\newcommand}
\newcommand{\delete}[1]{}

\nc{\mlabel}[1]{\label{#1}}  
\nc{\mcite}[1]{\cite{#1}}  
\nc{\mref}[1]{\ref{#1}}  
\nc{\meqref}[1]{~\eqref{#1}} 
\nc{\mbibitem}[1]{\bibitem{#1}} 

\delete{
\nc{\mlabel}[1]{\label{#1}  
{\hfill \hspace{1cm}{\bf{{\ }\hfill(#1)}}}}
\nc{\mcite}[1]{\cite{#1}{{\bf{{\ }(#1)}}}}  
\nc{\mref}[1]{\ref{#1}{{\bf{{\ }(#1)}}}}  
\nc{\meqref}[1]{~\eqref{#1}{{\bf{{\ }(#1)}}}} 
\nc{\mbibitem}[1]{\bibitem[\bf #1]{#1}} 
}

\newtheorem{thm}{Theorem}[section]
\newtheorem{lem}[thm]{Lemma}

\newtheorem{question}[thm]{Question}

\newtheorem{coro}[thm]{Corollary}
\theoremstyle{definition}
\newtheorem{defi}[thm]{Definition}

\newtheorem{rmk}[thm]{Remark}

\nc{\tred}[1]{\textcolor{red}{#1}}
\nc{\tblue}[1]{\textcolor{blue}{#1}}
\nc{\tgreen}[1]{\textcolor{green}{#1}}
\nc{\tpurple}[1]{\textcolor{purple}{#1}}
\nc{\btred}[1]{\textcolor{red}{\bf #1}}
\nc{\btblue}[1]{\textcolor{blue}{\bf #1}}
\nc{\btgreen}[1]{\textcolor{green}{\bf #1}}
\nc{\btpurple}[1]{\textcolor{purple}{\bf #1}}

\makeatletter

\theoremstyle{definition}

\newtheorem{remark}[thm]{Remark}

\nc{\name}[1]{{\bf #1}}

\nc{\ot}{\otimes}
\nc{\bfk}{\mathbf{k}}
\nc{\id}{\mathrm{id}}
\nc{\calf}{\mathcal{F}}
\nc{\calp}{\mathcal{P}}
\nc{\calq}{\mathcal{Q}}

\allowdisplaybreaks

\newcommand{\N}{\mathbb{N}}

\nc{\Res}{\mathrm{Res}}

\def \ra {\rightarrow}

\def\<{\langle}
\def\>{\rangle}

\nc{\mim}{\mathrm{im}}

\nc{\supp}{\text{supp}}
\nc{\Map}{\text{Map}}
\nc{\CSG}{\text{CSG}}

\nc{\fncmna}{$Nov_A^{NC}[X]$\xspace}	

\nc{\CM}{\text{CM}}		
\nc{\KCM}{\text{KCM}}	%

\nc{\mda}{multi-differential algebra}	
\nc{\mdas}{multi-differential algebras}	
\nc{\Mda}{Multi-differential algebra}	
\nc{\Mdas}{Multi-differential algebras}	

\nc{\cmdca}{commuting multi-differential commutative algebra\xspace}
\nc{\cmdnca}{commuting multi-differential noncommutative algebra\xspace}
\nc{\ncmdca}{noncommuting multi-differential commutative algebra\xspace}
\nc{\ncmdnca}{noncommuting multi-differential noncommutative algebra\xspace}

\nc{\cmdcas}{commuting multi-differential commutative algebras\xspace}
\nc{\cmdncas}{commuting multi-differential noncommutative algebras\xspace}
\nc{\ncmdcas}{noncommuting multi-differential commutative algebras\xspace}
\nc{\ncmdncas}{noncommuting multi-differential noncommutative algebras\xspace}

\nc{\Cmdca}{Commuting multi-differential commutative algebra\xspace}
\nc{\Cmdnca}{Commuting multi-differential noncommutative algebra\xspace}
\nc{\Ncmdca}{Noncommuting multi-differential commutative algebra\xspace}
\nc{\Ncmdnca}{Noncommuting multi-differential noncommutative algebra\xspace}

\nc{\Cmdcas}{Commuting multi-differential commutative algebras\xspace}
\nc{\Cmdncas}{Commuting multi-differential noncommutative algebras\xspace}
\nc{\Ncmdcas}{Noncommuting multi-differential commutative algebras\xspace}
\nc{\Ncmdncas}{Noncommuting multi-differential noncommutative algebras\xspace}

\nc{\cmdcac}{\rm{CMDCA}\xspace}	
\nc{\cmdncac}{\rm{CMDNCA}\xspace}	
\nc{\ncmdcac}{\rm{NCMDCA}\xspace}	
\nc{\ncmdncac}{\rm{NCMDNCA}\xspace} 

\nc\MNE{\text{MNE}_{\Omega}(\Omegax)}
\nc\fmna{\text{Nov}_{\Omega}(\Omegax)}
\nc\fmma{\text{Mag}_{\Omega}(\Omegax)}
\nc\fmdc{\text{CD}_{\Omega}(\Omegax)}
\nc\fmdca{\text{PCD}_{\Omega}(\Omegax)}

\nc\Omegax{X}
\nc\NE{\text{NE}(\Omegax)}
\nc\fnp{\triangleright} 
\nc\fna{\text{Nov}(\Omegax)}
\nc\fma{\text{Mag}(\Omegax)}
\nc\fdc{\text{CD}(\Omegax)}
\nc\fdca{\text{CD}(\Omegax)_0}
\nc{\PBTN}{\text{PBT}(\Omegax)}
\nc{\PMTN}{\text{PMT}(\Omegax)}
\nc{\BTN}{\text{BT}(\Omegax)}
\nc{\MTN}{\text{MT}(\Omegax)}

\nc{\Mag}{\text{Mag}}
\nc{\Nov}{\text{Nov}}
\nc{\I}{\text{I}}
\nc{\PBT}{\text{PBT}}
\nc{\PMT}{\text{PMT}}
\nc{\BT}{\text{BT}}
\nc{\MT}{\text{MT}}

\nc{\cmna}{multi-Novikov algebra\xspace}
\nc{\cmncna}{multi-noncommutative Novikov algebra\xspace}
\nc{\ncmna}{noncommuting multi-Novikov algebra\xspace}
\nc{\ncmncna}{noncommuting multi-noncommutative Novikov algebra\xspace}

\nc{\cmdas}{multi-differential algebras\xspace}
\nc{\cmnas}{multi-Novikov algebras\xspace}
\nc{\cmncnas}{multi-noncommutative Novikov algebras\xspace}
\nc{\ncmnas}{noncommuting multi-Novikov algebras\xspace}
\nc{\ncmncnas}{noncommuting multi-noncommutative Novikov algebras\xspace}

\nc{\Cmna}{Multi-Novikov algebra\xspace}
\nc{\Cmncna}{Multi-noncommutative Novikov algebra\xspace}
\nc{\Ncmna}{Noncommuting multi-Novikov algebra\xspace}
\nc{\Ncmncna}{Noncommuting multi-noncommutative Novikov algebra\xspace}

\nc{\Cmnas}{Multi-Novikov algebras\xspace}
\nc{\Cmncnas}{Multi-noncommutative Novikov algebras\xspace}
\nc{\Ncmnas}{Noncommuting multi-Novikov algebras\xspace}
\nc{\Ncmncnas}{Noncommuting multi-noncommutative Novikov algebras\xspace}

\nc{\mpre}{{\text{pre-}}}
\nc{\genpre}{{\text{pre-}}(W)}
\nc{\morpre}{\Phi}
\nc{\relpre}{\text{pre-}(S)}
\nc{\oppre}{\text{pre-}(\calq)}
\nc{\genpost}{\text{post-}(W)}
\nc{\morpost}{\Psi}
\nc{\relpost}{\text{post-}(S)}
\nc{\oppost}{\text{post-}(\calq)}

\nc{\mprer}{{\text{rpre-}}}
\nc{\genprer}{{\text{rpre-}}(W)}
\nc{\morprer}{\Phi_r}
\nc{\relprer}{\text{rpre-}(S)}
\nc{\opprer}{\text{rpre-}(\calq)}

\nc{\dfgen}{V} \nc{\dfrel}{\Lambda}
\nc{\dfgenb}{\vec{v}} \nc{\dfrelb}{\vec{r}}
\nc{\dfgene}{v} \nc{\dfrele}{r}

\nc{\obb}{\ \begin{picture}(-1,1)(-1,-3)\circle*{3}\end{picture}\ \,}

\nc{\odiff}{\mathcal{DA}} 
\nc{\oWNov}{\mathcal{Nov}_\lambda} 
\nc{\onov}{\mathcal{Nov}}	

\nc{\omwdiff}{{\mathcal{MD}iff_\Omega}}
\nc{\oncmdnca}{\mathcal{MDA}^{\mathrm{NC}}_{\mathrm{NC}}}
\nc{\ocmdnca}{\mathcal{MDA}^{\mathrm{NC}}}
\nc{\oncmdca}{\mathcal{MDA}_{\mathrm{NC}}}
\nc{\ocmdca}{\mathcal{MDA}}
\nc{\ad}{\text{ad}}

\nc{\cfreediff}{F_{\mathrm{Diff}}(X)}	
\nc{\ncfreediff}{F_{\mathrm{Diff}}^{\mathrm{NC}}(X)} 
\nc{\ncdiffpoly}{\bfk^{\mathrm{NC}}\{X\}}	

\nc{\omnov}{\mathcal{N}_\Omega}
\nc{\omdiff}{\mathcal{MD}iff}
\nc{\novop}{\rhd}

\nc{\freecmdiff}{F_{\mathrm{CDiff}}^{C}(X)}	
\nc{\cmdiffpoly}{\bfk_{\Omega}^{\mathrm{C}}\{X\}}	
\nc{\mapfinsupp}{\mathrm{Map}^{\mathrm{fs}}}

\nc{\NC}{\mathrm{NC}}
\nc{\ncmder}{\Delta_{\NC,\Omega}}

\nc{\comp}{induced\xspace}
\nc{\mcompst}{induced structure\xspace}
\nc{\mcompsts}{induced structures\xspace}
\nc{\Mcompst}{Induced structure\xspace}
\nc{\Mcompsts}{Induced structures\xspace}
\nc{\precompst}{pre-induced structure\xspace}
\nc{\precompsts}{pre-induced structures\xspace}
\nc{\postcompst}{post-induced structure\xspace}
\nc{\postcompsts}{post-induced structures\xspace}
\nc{\Precompst}{Pre-induced structure\xspace}
\nc{\Precompsts}{Pre-induced structures\xspace}
\nc{\Postcompst}{Post-induced structures\xspace}
\nc{\Postcompsts}{Post-induced structures\xspace}

\nc{\mcomp}{induced\xspace}
\nc{\mcomps}{induced\xspace}
\nc{\Mcomp}{Induced\xspace}
\nc{\Mcomps}{Induced\xspace}
\nc{\precomp}{pre-induced\xspace}
\nc{\precomps}{pre-induced\xspace}
\nc{\postcomp}{post-induced\xspace}
\nc{\postcomps}{post-induced\xspace}
\nc{\Precomp}{Pre-induced\xspace}
\nc{\Precomps}{Pre-induced\xspace}
\nc{\Postcomp}{Post-induced\xspace}
\nc{\Postcomps}{Post-induced\xspace}

\nc{\rmi}{\mathrm{I}}
\nc{\rmii}{\mathrm{II}}
\nc{\rmiii}{\mathrm{III}}
\nc{\oncmnov}{\mathcal{N}_{{\rm NC},\Omega}}
\nc{\sform}{\mathsf{S}}
\nc{\tform}{\mathsf{T}}
\nc{\quotcom}{\pi_{\rm com}}

\begin{document}

\title[Induced quadratic structures of operated algebras]{Induced quadratic structures of operated algebras with applications to multi-Novikov algebras}

\author{Li Guo}
\address{
Department of Mathematics and Computer Science,
Rutgers University,
Newark, NJ 07102, United States}
\email{liguo@rutgers.edu}

\author{Xiaoyan Wang}
\address{Department of School of Mathematical Sciences, East China Normal University, Shanghai 200241, China}
\email{wangxy@math.ecnu.edu.cn}

\author{Huhu Zhang}
\address{School of Mathematics and Statistics
Yulin University,
Yulin, Shaanxi 719000, China}
\email{huhuzhang@yulinu.edu.cn}

\date{\today}

\begin{abstract}
We formulate binary-quadratic induced structures for unary--binary operads by taking the arity-three part of the kernel of the substitution map for chosen induced binary operations. We use this to identify all universal relations in this quadratic class. For commutative associative algebras equipped with families of derivations, we determine these quadratic relations explicitly. In the noncommuting-derivation case they give the noncommuting multi-Novikov operad; imposing commutation of the derivations yields exactly the additional relations of the multi-Novikov operad. The one-derivation specialization recovers the usual Novikov operad. The calculations display how operator identities control induced binary-quadratic structures.
\end{abstract}

\subjclass[2020]{
18M60, 
17A30, 
18M70, 
17A50, 
12H05, 
17D25, 
17A36, 
16W25, 
16S10 
}

\keywords{operated algebra, operad, induced structure, Novikov algebra, differential algebra, multi-Novikov algebra, multi-differential algebra\\
Corresponding author: Li Guo}

\maketitle

\vspace{-1.2cm}

\tableofcontents

\vspace{-1.2cm}

\allowdisplaybreaks

\section{Introduction}
\mlabel{s:intro}

A recurring phenomenon in recently introduced algebraic structures is that they can naturally be induced by known algebraic structures equipped with certain linear operators. We develop an operadic notion of the \emph{binary-quadratic induced structure} from an operad with nonidentity unary operations, and use it to interpret multi-Novikov algebras. 

\subsection{Operated algebras}
\mlabel{ss:oprtalg}

Originated from diverse areas in mathematics, physics and engineering, the study of various linear operators on algebras has a long history. Well-known operators are endomorphisms, automorphisms and derivations. 
In the case of derivations, Ritt in his landmark work~\mcite{Ri1} a century ago introduced the concept of a differential algebra to provide an algebraic framework for differential equations. Abstracting from the derivation in analysis, a differential algebra is an (associative) algebra $R$ equipped with a linear operator $d$ that satisfies the following Leibniz rule.
\begin{equation}
d(xy)=d(x)y+xd(y), \quad x,\ y\in R.
\mlabel{eq:der10}
\end{equation}
Through the contributions of numerous mathematicians, including Kolchin, Kaplansky, and Singer, differential algebra has evolved into a substantial field encompassing differential Galois groups, differential algebraic groups, and differential algebraic geometry. This area has found broad applications in number theory, logic, and the mechanical proof of mathematical theorems~\mcite{Ko,PS,Wu2}.	

There are numerous other linear operators which are less well-known but nevertheless important and long-studied, such as Rota-Baxter operators from probability, Nijenhuis operators from geometry, averaging operators and Reynolds operators from fluid mechanics, with recent applications~\mcite{Bax,Bi,BHZ,CK,Da,GGL,HS,Ni,Re,Sh,TS}. 

These operator-acted structures are organized under the general term of an {\bf operated algebra}, defined to be an algebra equipped with certain type of linear operators~\mcite{Gop}, which can be tracked back to Kurosh's notion of $\Omega$-algebras\,\mcite{Ku}. 

\subsection{\Mcomp structures from operated algebras}
Compared to the traditional algebraic structures defined by binary or higher arity operations, systematic study of operated algebras poses new challenges. A fascinating pattern that have emerged over the years is the connection from an operated algebra to traditional algebraic structures, which has been generally called a \mcompst. The earliest instance was discovered by S.\,Gelfand that was recorded in\,\mcite{GD1}.
Recall that a {\bf Novikov algebra} is a nonassociative algebra $(R,\novop)$, 
where $\novop$ satisfies the identities 
\begin{eqnarray}\mlabel{eq:n1}
&(x\novop y)\novop z-x\novop(y\novop z)=(y\novop x)\novop z-y\novop(x\novop z)&
\\
&\mlabel{eq:n2}
(x\novop y)\novop z=(x\novop z)\novop y, \quad x, y, z\in R.&
\end{eqnarray}
Note that the first identity defines the pre-Lie algebra that is prominent in its own right\,\mcite{Bai,Man}. 

\begin{thm} \mlabel{t:gel}
$($Gelfand's Construction$)$
Let $(R,\cdot)$ be a commutative associative algebra equipped with a derivation $D$, that is $(R,\cdot,D)$ is a commutative differential algebra. Define
\begin{equation}
a\novop b:= a\cdot D(b), \ a, b\in R.
\mlabel{eq:diffprod}
\end{equation}
Then $(R,\novop)$ is a Novikov algebra. 
\end{thm}
Gelfand's construction is of fundamental importance in the theory of Novikov algebras. All complex Novikov algebras with dimensions no more than three and many important infinite-dimensional simple Novikov algebras can be realized by this construction or its linear deformations~\mcite{BM, BM1, X3}. It is eventually shown that every Novikov algebra is isomorphic to a subalgebra of some Novikov algebra obtained this way~\mcite{BCZ, DL}. Moreover, Gelfand's construction provides a natural right Novikov algebra structure on the vector space of Laurent polynomials~\mcite{HBG}, which is closely related with Novikov algebra affinization \mcite{BN}. An operadic study was carried out in\,\mcite{KSO}. 

With the renaissance of the operad theory and emergence of new algebraic structures, numerous instances of \mcompsts have appeared. For example, a Rota-Baxter Lie algebra induces a pre-Lie algebra or a post-Lie algebra; a Rota-Baxter associative algebra induces a dendriform algebra or a tridendriform algebra; an averaging associative (resp. commutative, resp. Lie) algebra induces a diassociative (resp. perm, resp. Leibniz) algebra; and a Nijenhuis associative algebra induces an NS-algebra and an N-dendriform algebra~\mcite{Ag,BBGN,BGN,Eb,GKo,Le,LG}. For algebras with multiple operators, matching Rota-Baxter associative (resp. Lie, resp. commutative) algebras of weight zero are shown to give rise to matching dendriform (resp. multi-pre-Lie, resp. matching Zinbiel) algebras that stemmed from regularity structure of stochastic PDEs\,\mcite{BHZ,Foi,GGZ19,Ha}. Recently, again from the regularity structures, the notion of multi-Novikov algebras was introduced and was shown to be induced from commutative algebras equipped with multiple commuting derivations~\mcite{BrD}. As related studies, induced structures from operators on Hopf algebras, Lie-type bialgebras, groups and braces have also appeared\,\mcite{BGLM,GGHZ,Go,GJSW,HBG,LST,LST2,ZLYG}.

While such phenomena are all called induced structures, their usage is quite informal, for example with no justification on why only the specific algebra should be called the particular induced structure. Take as an illustration the Novikov algebra which is regarded as the induced structure of the differential commutative algebra thanks to Gelfand's Theorem~\mref{t:gel}. Indeed the Novikov axioms are satisfied by the induced product in Eq.\,\meqref{eq:diffprod}. But by this standard, the pre-Lie algebra axiom \meqref{eq:n1} is also satisfied by the induced product and could also be called an induced structure of the differential commutative algebra. In the other direction, beyond the two Novikov algebra axioms, could there be a third independent identity that can also be induced from the differential commutative algebra? In fact, for the Nijenhuis associative algebra, an induced structure called NS algebra\,\mcite{Le} was introduced with four axioms. It was only later realized that there is a fifth axiom that is also satisfied by the induced products from Nijenhuis algebras\,\mcite{LG}. 
This motivates asking which \emph{binary-quadratic} identities are universally satisfied by the chosen derived operations, and whether the familiar axioms exhaust that specific class of identities.
More precisely, we ask

\begin{question} 
{\rm (Quadratic converse of Gelfand's construction)} Are the Novikov identities all the universal binary-quadratic identities of $x\novop y=xD(y)$ on commutative differential algebras?
\mlabel{q:daconv}
\end{question}

For general operated algebras, there are the following natural questions.

\begin{question}
What is the binary-quadratic induced structure associated with a chosen family of derived binary operations?
\mlabel{q:what}	
\end{question}

\begin{question}
How can one explicitly determine its complete arity-three relation space?
\mlabel{q:how}
\end{question}

We answer these questions. Section~\mref{s:ind} defines the pre- and post-induced binary-quadratic operads by intersecting the substitution kernel with arity three. Sections~\mref{s:ncdifcaind} and~\mref{sec:diffcomp} compute the resulting relation spaces directly: the noncommuting multi-Novikov structure (as in~\mcite{WGZ1}) arises for noncommuting derivations, while the commuting case gives the multi-Novikov structure introduced by Bruned and Dotsenko~\mcite{BrD}. In particular, the usual Novikov identities form the complete binary-quadratic relation space for one derivation. 

Even though the framework introduced here is for induced binary-quadratic identities, it applies to many induced structures such as those mentioned above. More general induced identities will be studied separately. 

\subsection{Layout of the paper}
\mlabel{ss:layout}

This paper is organized as follows. 	

In Section~\mref{s:ind}, we work in the operadic framework and give a rigorous definition of the binary quadratic operad that can be induced from an operated algebra. These are called the pre-induced and post-induced \emph{binary-quadratic} structures, always relative to the chosen operations.

In Section~\mref{s:ncdifcaind}, we first determine the \precompst of the noncommuting multi-differential commutative algebra.  The proof is organized by arity-three normal forms and uses the free noncommuting multi-differential algebra from Section~\mref{ss:mdas} only to establish their linear independence.  In Section~\mref{sec:diffcomp}, the commuting case is then obtained from the noncommuting one by quotienting by the commutators of the derivations.  This quotient produces exactly the additional multi-Novikov relation.  In particular, the one-derivation specialization gives Corollary~\mref{co:novind}; see Remark~\mref{r:novind}. 

\section{\Mcompsts of operated algebras and operated operads}
\mlabel{s:ind}
This section first recalls the needed background on operads and then formulates the \mcompst of an operated algebra in the context of operads. 

\subsection{Background and notions on operads}
We first give notions and properties of a binary quadratic symmetric operad for later application. See~\mcite{LV} for details. 
\begin{defi}
\begin{enumerate}
\item
An {\bf $\mathbb{S}$-module} over $\bfk$ is a family
$$M :=\{M(n)\}_{n\geq0}=\{M(0),M(1),\ldots 
M(n),\ldots \}$$
of right $\bfk S_n$-modules $M{(n)}$, where $S_n$ is the symmetric group for $n\geq0$. If each $M(n)$ is the trivial $\bfk S_n$-module, 
that is, $M(n)$ is a $\bfk$-module, then $M$ is called an {\bf $\mathbb{N}$-module}.
\item
A {\bf nonsymmetric (linear) operad } is an $\mathbb{N}$-module $\calp=\{\calp(n)\}_{n\geq0}$ equipped with {\bf partial compositions}:
\begin{equation}
\circ_i:=\circ_{m,n,i}: \calp(m)\ot \calp(n)\longrightarrow \calp(m+n-1), \quad 1\leq i\leq m,
\mlabel{eq:opc}
\end{equation}
such that, for $\eta\in\calp(\ell), \mu\in\calp(m)$ and $\nu\in\calp(n)$, the following relations hold.
\begin{enumerate}
\item $
(\eta \circ_i \mu)\circ_{i-1+j}\nu = \eta\circ_i (\mu\circ_j\nu), \quad 1\leq i\leq \ell, 1\leq j\leq m.$
\mlabel{it:esc}
\item$(\eta\circ_i\mu)\circ_{k-1+m}\nu =(\eta\circ_k\nu)\circ_i\mu, \quad
1\leq i<k\leq \ell.$
\mlabel{it:epc}
\item
There is an element $\id\in \calp_1$ such that $\id\circ \mu=\mu$ and $\mu\circ\id=\mu$ for $\mu\in \calp_n, n\geq 0$.
\mlabel{it:id}
\end{enumerate}
\item A {\bf (symmetric linear) operad} is an $\mathbb{S}$-module $\calp=\{\calp(n)\}_{n\geq0}$   equipped with nonsymmetric linear operad structure, that satisfies the following additional conditions.
\begin{enumerate} 
\item 
For $\sigma\in S_m$, we have:
\begin{equation}
\mu^\sigma \circ_i \nu = (\mu\circ_{\sigma(i)} \nu)^{\sigma'}, \quad 1\leq i\leq m, 
\mlabel{eq:symop2}
\end{equation}
Here $\sigma'$ is the permutation on $m+n-1$ inputs induced by moving the block of $n$ consecutive inputs in slot $i$ to slot $\sigma(i)$, simultaneously moving the other singleton blocks according to $\sigma$, while preserving the order of the inputs inside the length-$n$ block. Thus, for $m=n=2$, $i=1$ and $\sigma=(12)$, the induced permutation is $\sigma'=(123)$.
\item For $\tau\in S_n$, we have 
\begin{equation}
\mu \circ_i \nu^\tau = (\mu\circ_i\nu)^{\tau'}, 1\leq i\leq m, 
\mlabel{eq:symop1}
\end{equation}
where $\tau'\in S_{m-1+n}$ is the permutation which acts by the identity, except on the block $\{i,\ldots,i-1+n\}$ on which it acts via $\tau$. 
\end{enumerate}
\end{enumerate}
\mlabel{d:operad}
\end{defi}

A {\bf morphism of operads} from $\calp$ to $\calq$ is a family of $S_n$-equivariant maps $f_n: \calp(n)\to \calq(n)$ that is compatible with the compositions of operads.
An operad $\mathscr{P}$ is called {\bf reduced} if $\mathscr{P}(0) = 0$.

{\em In this paper, all operads are assumed to be reduced.}

For a vector space $V$, the {\bf endomorphism operad} ${\rm End}_V$ is defined by setting $\operatorname{End}_V(n)=\operatorname{Hom}(V^{\otimes n},V)$, for  $n\ge 1$; for reduced operads we take $\operatorname{End}_V(0)=0$.
The action of $S_n$ on ${\rm End}_V(n)$ is given by the right action
\[(f\cdot\sigma)(v_1,\ldots,v_n):=f(v_{\sigma^{-1}(1)},\ldots,v_{\sigma^{-1}(n)}).\]

The composition map $\gamma$ is given by the usual composition of multivariate functions.
For this example, the maps $\sigma'$ and $\tau'$ in Definition~\mref{d:operad} are the natural permutations on $\mathrm{End}_V(m+n-1)$ that are induced from $\sigma$ on $\mathrm{End}_V(m)$ and $\tau$ on $\mathrm{End}(n)$ respectively. 

Let $\calp$ be an operad. A {\bf $\calp$-algebra} is a vector space $V$ with a morphism $\rho: \calp\to {\rm End}_V$ of operads. We say that the operad $\calp$ {\bf encodes} $\calp$-algebras.

A class of operads can be given explicitly as follows~\cite[Sec. 7.6.2]{LV}.
Let $V$ be an $\mathbb S$-module concentrated in arity two, and also write $V=V(2)$. The free reduced operad $\calf(V)$ contains the operadic identity $\id$ in arity one, supplied by the free-operad construction, not by $V(1)$. Its arity-$n$ component is represented by binary trees with $n$ labeled leaves and vertices decorated by $V$. An arbitrary binary operad is a quotient
$$\calp(V,R):=\calf(V)/(R)$$
for some $S_2$-module $V$ and an operadic ideal $(R)$ of $\calf(V)$ generated by a subset $R\subseteq \calf(V)$. 
$V$ is called the {\bf space of generators} and $R$ is called the {\bf space of relators}. If all elements of $R$ are quadratic, namely linear combinations of the form $\mu\circ_i \nu$ for $\mu, \nu\in V, i=1,2$, then the operad $\calf(V,R)$ is called {\bf binary quadratic}. 

To determine the possible quadratic relators for $V$, we note that 

\begin{equation}
\calf(V)(3)\cong 3(V\ot V)=(V\ot V)\oplus (V\ot V)\oplus (V\ot V)
=(V\circ_\rmi V) \oplus (V\circ_\rmii V) \oplus (V\circ_\rmiii V).
\mlabel{eq:quadrel}
\end{equation}
This is a vector-space decomposition (with the $S_3$-action obtained by permuting the leaves), not a claim that each of the three summands is an $S_3$-submodule.
Here for $\mu, \nu\in V$, we follow the notations in  \cite[Sec 7.6.2]{LV}:
\begin{align}
(\mu\circ_\rmi \nu)(x,y,z)&:=\mu(\nu(x,y),z), \mlabel{eq:comp1}\\
(\mu\circ_\rmii \nu)(x,y,z)&:=\mu(\nu(y,z),x), \mlabel{eq:comp2}\\
(\mu\circ_\rmiii \nu)(x,y,z)&:=\mu(\nu(z,x),y).
\mlabel{eq:comp3}
\end{align}

In the special case when $V=\bfk \{\mu,\mu'\}$, the notations can be further simplified to match with the classical multiplications, as in \cite[p. 247]{LV}. Let $\sigma=(12)\in S_2$ be the unique nonunit permutation, and denote $\mu'=\mu^\sigma$, the opposite product of $\mu$.
Then we have the following table, where $\mu(x,y)$ is abbreviated to $xy$. 
\begin{equation}
\begin{tabular}{rlrlrl}\hline
1 & $\mu\circ_\rmi \mu \leftrightarrow (xy)z$ &\quad 5& $\mu\circ_\rmiii\mu \leftrightarrow (zx)y$ &\quad  9 & $\mu\circ_\rmii \mu \leftrightarrow (yz)x$ 	\\
2 & $\mu'\circ_\rmii \mu \leftrightarrow x(yz)$ &\quad 6& $\mu'\circ_\rmi\mu \leftrightarrow z(xy)$ &\quad  10 & $\mu'\circ_\rmiii \mu \leftrightarrow y(zx)$\\
3 & $\mu'\circ_\rmii \mu' \leftrightarrow x(zy)$ &\quad 7& $\mu'\circ_\rmi\mu' \leftrightarrow z(yx)$ &\quad  11 & $\mu'\circ_\rmiii \mu' \leftrightarrow y(xz)$\\
4 & $\mu\circ_\rmiii \mu' \leftrightarrow (xz)y$ &\quad 8& $\mu\circ_\rmii\mu' \leftrightarrow (zy)x$ &\quad  12 & $\mu\circ_\rmi \mu' \leftrightarrow (yx)z$
\\	\hline
\end{tabular}
\mlabel{eq:bqtable}
\end{equation}

\subsection{The \mcompst of an operated operad}

A {\bf unary-binary operad} (also called an {\bf operated binary operad}) is an operad \emph{generated} by unary and binary operations~\mcite{ZGG}. More precisely, it is a quotient of $\calf(W)$ for a generator module $W=W_1\oplus W_2$ concentrated in arities one and two. The identity of $\calf(W)$ is supplied separately from $W_1$. Relations may subsequently identify a unary generator with the operadic identity. Such an operad is of the form
$$\calq:=\calp(W,S)=\calf(W)/(S),$$ 
where $(S)$ is the operadic ideal generated by a subset $S\subseteq\calf(W)$.

Let $\Omega$ be a linear basis of $W_1$. Choose a linear basis $\Theta$  of $W_2$ such that for any $\theta\in \Theta$, $\theta^{(12)}$ is also in $\Theta$.
Define the $\mathbb{S}$-module concentrated at degree two: 
$$\genpre:=\genpre_{\Omega,\Theta}=\bfk\{\prec_{\omega,\theta},\succ_{\omega,\theta}\,|\,\omega\in \Omega, \theta\in \Theta\},$$
where the $S_2$-action on $\genpre$ is given by
$$\prec_{\omega,\theta}^{(12)}:=\succ_{\omega,\theta^{(12)}} ,\quad \succ_{\omega,\theta}^{(12)}:=\prec_{\omega,\theta^{(12)}},\quad\omega\in\Omega,\theta\in\Theta.$$
Define an operad homomorphism, called the {\bf \precomp homomorphism}  
\begin{equation} \mlabel{eq:geninddir}
\morpre:=\morpre_{\calq}: \calf(\genpre)\to \calf(W), \quad
\left\{\begin{array}{l} \prec_{\omega,\theta}\mapsto\theta(\id\ot \omega) , \\ \succ_{\omega,\theta}\mapsto\theta(\omega\ot \id) ,
\end{array}\right. \theta\in \Theta, \omega\in \Omega.
\end{equation}
Denote 
\begin{equation} \mlabel{eq:genindreldi}
\relpre:=\relpre{}_\calq:=\ker \pi\morpre \cap \calf(\genpre)(3),
\end{equation}
where $\calf(\genpre)(3)=3\,\Big(\genpre^{\ot 2}\Big)$ with the abbreviation in Eq.~\meqref{eq:quadrel} and 
$\pi$ is the quotient map from $\calf(W)$ to $\calq$.
\begin{defi} \mlabel{de:preind}
Let $\calq=\calp(W,S)$ be a unary-binary operad. 
The {\bf \precomp binary quadratic operad} or simply the {\bf \precomp operad} 
of $\calq$ is defined to be  the binary quadratic operad given by the quotient \begin{equation}\mlabel{eq:genindquotdi} \oppre:=\calf(\genpre)/(\relpre).
\end{equation}
The $\oppre$-algebra is called the {\bf \precompst} of the $\calq$-algebra.
\end{defi}
More precisely, restriction of the structure map along $\oppre\to\calq$ equips every $\calq$-algebra with an $\oppre$-algebra structure. The quotient above imposes the complete arity-three kernel, not the entire operadic kernel of $\morpre$.

The related operads satisfy the following commutative diagram.
$$\xymatrix{
\calf(\genpre)\ar@{->}[r]^{\morpre}\ar@{->}[dr]^{\pi\morpre}\ar@{->>}[d]&\calf(W)\ar@{->>}[d]^{\pi}\\
\oppre:=\calf(\genpre)/(\relpre)\ar@{->}[r]&\calf(W)/(S)=:\calq
}$$

\begin{remark}
\mlabel{r:indall}
The space $\relpre$ consists precisely of the universal \emph{arity-three} identities among the selected pre-induced binary operations. The ideal $(\relpre)$ contains their formal operadic consequences in all arities. In general, it need not equal $\ker\pi\morpre$. 

We retain the chosen formal binary signature. If $\pi\morpre$ has a nonzero kernel on arity two, such linear dependencies are not imposed by the present quadratic quotient; this does not affect the characterization of its arity-three kernel. One may instead first pass to the image of the binary generator module, but that defines a different chosen signature and must be indicated explicitly.
\end{remark}

Similarly, define the $\mathbb{S}$-module concentrated at degree two.
The post-induced binary signature retains the original binary operations independently of the chosen unary operators. Explicitly, set
\[
\genpost:=\genpost_{\Omega,\Theta}
=\genpre_{\Omega,\Theta}\oplus
\bfk\{\obb_\theta\mid\theta\in\Theta\}.
\]
The $S_2$-action on $\genpost$ is given by
$$\prec_{\omega,\theta}^{(12)}:=\succ_{\omega,\theta^{(12)}} ,\quad \succ_{\omega,\theta}^{(12)}:=\prec_{\omega,\theta^{(12)}},\quad\obb_{\theta}^{(12)}:=\obb_{\theta^{(12)}},\quad \omega\in\Omega,\theta\in\Theta.$$
Next, define the operad homomorphism, called the {\bf \postcomp homomorphism}  
\begin{equation} \mlabel{eq:genindtri}
\morpost:=\morpost_{\calq}: \calf(\genpost)\to \calf(W), \quad
\left\{\begin{array}{l} \prec_{\omega,\theta}\mapsto\theta(\id\ot \omega) , \\ \succ_{\omega,\theta}\mapsto \theta(\omega\ot \id),\\
\obb_{\theta}\mapsto \theta,
\end{array}\right. \theta\in \Theta, \omega\in \Omega.
\end{equation}

Denote 
\begin{equation} \mlabel{eq:genindreltri}
\relpost:=\relpost{}_\calq:=\ker \pi\morpost \cap \calf(\genpost)(3),
\end{equation}
where $\calf(\genpost)(3)=3\,\Big(\genpost^{\ot 2}\Big)$ with the abbreviation in Eq.~\meqref{eq:quadrel}.

\begin{defi} \mlabel{de:postind}
The {\bf \postcomp binary quadratic operad} or simply the {\bf \postcomp operad} of the unary-binary operad $\calq=\calp(W,S)$ is defined to be the binary quadratic operad given by the quotient 
\begin{equation}\mlabel{eq:genindquottri} \oppost:=\calf(\genpost)/(\relpost).
\end{equation}
The $\oppost$-algebra is called the {\bf \postcompst} of the $\calq$-algebra.
\end{defi}
The same arity-three restriction applies to $\relpost$: the post-induced quadratic operad is the quotient by $(\relpost)$, not in general by the full kernel of $\pi\morpost$.

The \precomp and \postcomp operads satisfy the following commutative diagram.
$$\xymatrix{
\calf(\genpre)\ar@{->}[r]^{\morpre}\ar@{->}[dr]^{\pi\morpre}\ar@{->>}[d]&\calf(W)\ar@{->>}[d]^{\pi}& \calf(\genpost)\ar@{->}[l]_{\morpost}\ar@{->}[dl]_{\pi\morpost}\ar@{->>}[d]\\
\oppre:=\calf(\genpre)/(\relpre)\ar@{->}[r]&\calq=\calf(W)/(S)& \oppost:=\calf(\genpost)/(\relpost)\ar@{->}[l]
}$$

\begin{rmk} \mlabel{rk:compstr}
\begin{enumerate}
\item 
The inclusion of the pre-induced binary signature into the post-induced signature induces a canonical morphism $\oppre\to\oppost$: the arity-three kernels agree on the pre-induced signature. Conversely, setting $\obb_\theta=0$ in the free post-induced signature need not descend to a map $\oppost\to\oppre$. It does so if and only if the resulting free-operad projection sends $\relpost$ into the ideal $(\relpre)$. In general, no such reverse map is asserted.
\item
The case of nonsymmetric operads has been formulated and applied to the case of extended Rota-Baxter algebras in~\mcite{ZHG}. 
\item 
The \precomp and \postcomp binary quadratic operads hereby introduced can be regarded as the first stage of a more in-depth study of the induced structures of algebraic structures equipped with linear operators. 
\end{enumerate}
\end{rmk}

\subsection{Various multi-differential algebras and their operads}
\mlabel{ss:mdas}

We first recall the following notions from~\mcite{Gop,Gub}
\begin{defi}
\label{de:mapset}
An {\bf operated $\bfk$-algebra} is a $\bfk$-algebra $R$ together with a $\bfk$-linear map $P:R\to R$.  
\end{defi}

We also recall multi-differential commutative algebras~\mcite{Ko,PS} and their noncommuting analog. 

\begin{defi}
Let $\Omega$ be a nonempty set. Denote $\partial_\Omega:= \{\partial_\omega|\omega \in \Omega\}$.
\begin{enumerate} 
\item A \textbf{\cmdca} indexed by $\Omega$ is a commutative associative algebra $R$ equipped with derivations  $\partial_\Omega$ which are pairwise commuting. 
\item A \textbf{\ncmdca} indexed by $\Omega$ is a commutative associative algebra $R$ equipped with derivations $\partial_\Omega$ which are not necessarily pairwise commuting.  	
\end{enumerate}
\mlabel{d:mdas}
\end{defi}

We work over the stated ground field $\bfk$; the arity-three kernel arguments use no division and do not require characteristic zero. The polynomial objects below are free \emph{unital} multi-differential commutative algebras in the unital category. Their augmentation ideals are free in the nonunital category, and either polynomial algebra, considered without a specified unit, is a valid test algebra for multilinear identities.
For later applications, we also recall the construction of free objects for these multi-differential algebras \cite{WGZ1}.

For a set $Y$, let $M(Y)$ denote the free commutative monoid generated by $Y$, realized as the set of maps with finite supports:
\begin{equation}
M(Y):=\big\{\alpha:Y\to \N\,\big|\, |\supp(\alpha)|<\infty\big\}.
\end{equation}
Here the support of $\alpha$ is $\supp(\alpha):=\{y\in Y\,|\,\alpha(y)\neq 0\}$. 
The multiplication in $M(Y)$ is given by 
$$ (\alpha\cdot \beta)(y):=\alpha(y)+\beta(y), \quad y\in Y.$$
For convenience, $\alpha\in M(Y)$ is identified with $$\prod_{y\in \supp(\alpha)}y^{\alpha(y)}=y_1^{\alpha(y_1)}\cdots y_k^{\alpha(y_k)}$$ 
for $\supp(\alpha)=\{y_1,\ldots,y_k\}$, giving the usual construction of the free commutative monoid: 
\begin{equation}
M(Y):=\big\{ y_1^{\alpha_1}\cdots y_m^{\alpha_m}\,\big|\, y_i\in Y, \alpha_i\geq 1,m\geq 1\big\} \cup \{1\}.
\end{equation}

On the other hand, let $M_{\NC}(Y)$ be the free monoid on $Y$ realized as words:
\begin{equation}
M_{\NC}(Y):=\big\{y_1\cdots y_k\,\big|\, y_i\in Y, 1\leq i\leq k, k\geq 1\big\}\cup \{1\}
\mlabel{eq:ncmonoid}
\end{equation}
with the concatenation multiplication. 

Now for nonempty sets $\Omega$ and $X$, define the set $d_{\Omega}:=\{d_\omega\,|\,\omega\in \Omega\}$ and define the set of {\bf commuting multi-differential variables} to be 
\begin{equation}
\Delta_{\Omega}(X):=M(d_\Omega)\times X=\bigg\{\Big( \prod_{\omega\in \supp(\alpha)}d_\omega^{\alpha(\omega)},x\Big)~\bigg|~\alpha\in M(d_\Omega), x\in X\bigg\}.
\mlabel{eq:mdiffvar}
\end{equation}
As abbreviations, we write 
$$ \prod_{\omega\in \supp(\alpha)}d_\omega^{\alpha(\omega)}(x)=\Big(\prod_{\omega\in \supp(\alpha)}d_\omega^{\alpha(\omega)}\Big)(x)=\Big( \prod_{\omega\in \supp(\alpha)}d_\omega^{\alpha(\omega)},x\Big).$$

For $\omega\in \Omega$, define 
$$D_{\omega}: \Delta_{\Omega}(X)\ra\Delta_{\Omega}(X)$$ 
by 
$$D_\omega \Big(\prod_{\tau\in \supp(\alpha)}d_\tau^{\alpha(\tau)}(x)\Big):=d_\omega\prod_{\tau\in \supp(\alpha)}d_{\tau}^{\alpha(\tau)}(x)=d_{\omega}^{\alpha(\omega)+1}\prod_{\tau\in\supp(\alpha),\tau\neq\omega}d_{\tau}^{\alpha(\tau)}(x).$$ 
Take the polynomial algebra 
\begin{equation}
\bfk_\Omega\{X\}:=\bfk[\Delta_{\Omega}(X)]:=\bfk M(\Delta_{\Omega}(X)).
\end{equation}
Extend $D_\omega$ by applying the Leibniz rule repeatedly and the linearity to $D_\omega:\bfk_\Omega\{X\} \ra \bfk_\Omega\{X\}$. Let $i:X\ra \bfk_\Omega\{X\}, x\mapsto (1,x)$ be the natural injection. Then a fundamental result on differential algebra is 
\begin{thm} 
\mcite{Ko,PS}
The pair 	$(\bfk_\Omega\{X\}, \{D_\omega\}_{\omega\in\Omega})$ defined as above with $i:X\ra \bfk_\Omega\{X\}$ is the free \emph{unital} \cmdca\ on $X$.
\mlabel{t:fcmdca}
\end{thm}

We also consider the case when the derivations are not commuting.
Define the set of {\bf noncommuting multi-differential variables} to be 
\begin{equation}
\ncmder(X):=M_{\NC}(d_\Omega)\times X
\mlabel{eq:ncmdiffvar}
\end{equation}
For $\omega\in \Omega$, define 
\begin{equation}
D_{\omega}: \ncmder(X) \ra \ncmder(X), \quad D_\omega \big((d_{\omega_1}\cdots d_{\omega_k})(x)\big):=(d_\omega d_{\omega_1}\cdots d_{\omega_k})(x).
\mlabel{eq:deronmulti}
\end{equation}
Again take the polynomial algebra 
$$\bfk_{{\rm NC},\Omega}\{X\}:=\bfk [\ncmder(X)]:=\bfk M(\ncmder(X))$$
as in Theorem~\mref{t:fcmdca}. Extend $D_\omega$ for $\omega\in \Omega$ by the Leibniz rule to a linear operator on $\bfk_{{\rm NC},\Omega}\{X\}$. Still let $i: X\ra \bfk_{{\rm NC},\Omega}\{X\}$ be the natural injection.

\begin{thm} \mcite{WGZ1}
The pair $(\bfk_{{\rm NC},\Omega}\{X\}, \{D_\omega\}_{\omega\in\Omega})$ with $i: X\ra \bfk_{\NC,\Omega}\{X\}$ is the free \emph{unital} \ncmdca\ on $X$.
\mlabel{t:fncmdca}
\end{thm}

We now give an operadic interpretation of these multi-differential algebras.
Here the specified generator module is
\[W(1)=\bigoplus_{\omega\in\Omega}\bfk\rho_\omega,\qquad
W(2)=\bfk\mu,\qquad \mu^{(12)}=\mu,\qquad
W(n)=0\quad(n\ne1,2).\]
The operadic identity is not an additional unary generator; all $\rho_\omega$ are selected as the unary operators. The pre-induced binary signature then has exactly the two transposed generators $\novop_\omega$ and $\novop'_\omega$ for each $\omega$.
\begin{defi}
\begin{enumerate}
\item The {\bf operad of \ncmdcas} is 
\begin{equation}
\oncmdca:= \calp(W,\dfrel_{\oncmdca}), \mlabel{eq:oncmdca}
\end{equation}	
where
\begin{equation} 
\dfrel_{\oncmdca}:=\big\{\mu (\mu\ot \id)-\mu(\id\ot \mu),	\rho_\omega \mu-\mu(\id\ot \rho_\omega)-\mu(\rho_\omega\ot \id) \,\big|\, \omega \in \Omega \big\}
\mlabel{eq:oncmdcarel}
\end{equation}	
\item The {\bf operad $\ocmdca$ of \cmdcas} is 
\small{			\begin{equation}
\ocmdca:= \calp(W,\dfrel_{\ocmdca}), \quad \dfrel_{\ocmdca}:=\dfrel_{\oncmdca}\cup \big\{ \rho_\alpha \rho_\beta - \rho_\beta \rho_\alpha\,\big|\, \alpha, \beta \in \Omega\big\}.
\mlabel{eq:ocmdca}
\end{equation}	
}
\end{enumerate}
\mlabel{df:operads}
\end{defi}

\section{The \precompst of \ncmdcas}
\mlabel{s:ncdifcaind}

As a first application, we determine the complete \emph{quadratic} pre-induced relation space of \ncmdcas. The derivations need not commute, while the underlying associative multiplication remains commutative. We do not consider noncommutative associative base multiplication in this paper.  We start with the noncommuting case because the words in the
operators retain their order.  This allows us to see directly which
arity-three expressions can have the same image.  In
Section~\mref{sec:diffcomp}, the commuting case will then be obtained by
passing to the quotient in which the derivations commute.

\begin{defi}\cite{WGZ1}
\mlabel{d:ncmna}
A {\bf \ncmna} is a vector space $\text{N}$ equipped with bilinear products
$\rhd_\omega$ indexed by a set $\Omega$, which satisfy
\begin{eqnarray}
&(x \rhd_\omega y)\rhd_\tau z-x\rhd_\omega (y\rhd_\tau z)
=(y\rhd_\omega x)\rhd_\tau z-y\rhd_\omega (x\rhd_\tau z),&
\mlabel{eq:ncmn1}
\\
&(x\rhd_\omega y)\rhd_\tau z=(x\rhd_\tau z)\rhd_\omega y,
\quad x,y,z\in \text{N},\ \omega,\tau\in \Omega.&
\mlabel{eq:ncmn2}
\end{eqnarray}
\end{defi}

\begin{defi}
Let $\Omega$ be a set and
$V=\bfk\{\novop_\omega,\novop_\omega'\mid \omega\in\Omega\}$ be the
$S_2$-module, regarded as an $\mathbb S$-module concentrated in degree $2$,
with $S_2$-action $\novop_\omega^{(12)}=\novop_\omega'$.  Define
\begin{equation}
\begin{split}
{\dfrel}_{\oncmnov}:=\operatorname{span}_{\bfk}
\big\{r\cdot\sigma\mid{}&\sigma\in S_3,\ r=A_{\alpha,\beta}
 \text{ or }B_{\alpha,\beta},\ \alpha,\beta\in\Omega\big\},\\
A_{\alpha,\beta}:={}&
 \novop_\beta\circ_\rmi\novop_\alpha
 -\novop'_\alpha\circ_\rmii\novop_\beta
 -\novop_\beta\circ_\rmi\novop'_\alpha
 +\novop'_\alpha\circ_\rmiii\novop'_\beta,\\
B_{\alpha,\beta}:={}&
 \novop_\beta\circ_\rmi\novop_\alpha
 -\novop_\alpha\circ_\rmiii\novop'_\beta.
\end{split}
\mlabel{eq:ncmnovrel}
\end{equation}
The quotient operad
$\oncmnov:=\calp(V,\dfrel_{\oncmnov})$ is called the {\bf operad of
\ncmnas}.
\mlabel{d:oncmnov}
\end{defi}

The proof below uses a basis of $\calf(V)(3)$ adapted to the two ways in
which two derivations can be distributed among three variables.  For
$\alpha,\beta\in\Omega$ and three distinct placeholders $u,v,w$, put
\begin{align}
\sform_{\alpha,\beta}(u;v,w)
&:=(u\rhd_\alpha v)\rhd_\beta w,
\mlabel{eq:Sform}\\
\tform_{\alpha,\beta}(u,v;w)
&:=u\rhd_\alpha(v\rhd_\beta w)
 -(u\rhd_\alpha v)\rhd_\beta w.
\mlabel{eq:Tform}
\end{align}
The point of subtracting the left-associated term in
Eq.~\meqref{eq:Tform} is that the Leibniz-rule contribution in which the two
derivations hit different variables cancels.  Thus, after applying the
induced product $u\rhd_\omega v=uD_\omega(v)$, the $\sform$-terms correspond
to derivative distribution $(0,1,1)$ and the $\tform$-terms to derivative
distribution $(0,0,2)$.

\begin{lem}
\mlabel{lem:STbasis}
Fix three symbols $x_1,x_2,x_3$.  The following set $\mathcal B_{\sform}\sqcup\mathcal B_{\tform}$ is a basis of
$\calf(V)(3)$.
\begin{equation}
\begin{split}
\mathcal B_{\sform}
&:=\big\{\sform_{\alpha,\beta}(x_i;x_j,x_k)
  \mid (i,j,k)\text{ is a permutation of }(1,2,3),\
  \ \alpha,\beta\in\Omega\big\},\\
\mathcal B_{\tform}
&:=\big\{\tform_{\alpha,\beta}(x_i,x_j;x_k)
  \mid (i,j,k)\text{ is a permutation of }(1,2,3),\
  \ \alpha,\beta\in\Omega\big\}.
\end{split}
\mlabel{eq:STbasis}
\end{equation}
More precisely, every element of $\calf(V)(3)$ has a unique finite linear
expansion in the elements of
$\mathcal B_{\sform}\sqcup\mathcal B_{\tform}$.
\end{lem}

\begin{proof}
Since $V=\bigoplus_{\omega\in\Omega}\bfk S_2$, the formal planar monomials
\[
(x_i\rhd_\alpha x_j)\rhd_\beta x_k,
 \qquad x_i\rhd_\alpha(x_j\rhd_\beta x_k)
\]
for all ordered pairs $(\alpha,\beta)$ and all permutations $(i,j,k)$, form a vector-space basis of $\calf(V)(3)$. The following change of basis is made globally, not separately in fictitious fixed-label summands.
For every ordered triple $(i,j,k)$, keep the left-associated monomial
\[
 (x_i\rhd_\alpha x_j)\rhd_\beta x_k
 =\sform_{\alpha,\beta}(x_i;x_j,x_k),
\]
and replace the corresponding right-nested monomial
$x_i\rhd_\alpha(x_j\rhd_\beta x_k)$ by
\[
 \tform_{\alpha,\beta}(x_i,x_j;x_k)
 =x_i\rhd_\alpha(x_j\rhd_\beta x_k)
  -\sform_{\alpha,\beta}(x_i;x_j,x_k).
\]
This is an invertible triangular change of basis, since conversely
\[
 x_i\rhd_\alpha(x_j\rhd_\beta x_k)
 =\tform_{\alpha,\beta}(x_i,x_j;x_k)
  +\sform_{\alpha,\beta}(x_i;x_j,x_k).
\]
Taking the union over all ordered pairs $(\alpha,\beta)$ gives
Eq.~\meqref{eq:STbasis}.  If $\Omega$ is infinite, every element of the free
operad involves only finitely many labels, so the same argument applies on a
finite subset of $\Omega$ containing those labels.
\end{proof}

The two defining relations of a noncommuting multi-Novikov algebra have a
particularly simple form in this basis.

\begin{lem}
\mlabel{lem:NCnormal}
The subspace $\dfrel_{\oncmnov}\subseteq\calf(V)(3)$ is generated by
\begin{align}
\sform_{\alpha,\beta}(x_i;x_j,x_k)
 -\sform_{\beta,\alpha}(x_i;x_k,x_j),
\mlabel{eq:SsymNC}\\
\tform_{\alpha,\beta}(x_i,x_j;x_k)
 -\tform_{\alpha,\beta}(x_j,x_i;x_k),
\mlabel{eq:TsymNC}
\end{align}
where $(i,j,k)$ runs through the permutations of $(1,2,3)$ and
$\alpha,\beta\in\Omega$.  Consequently, a basis of
$\calf(V)(3)/\dfrel_{\oncmnov}$ is obtained by choosing one representative
from each orbit of $\mathcal B_{\sform}$ under
\begin{equation}
 (\alpha,\beta;i,j,k)\longleftrightarrow
 (\beta,\alpha;i,k,j)
\mlabel{eq:Sinvolution}
\end{equation}
and one representative from each orbit of $\mathcal B_{\tform}$ under
\begin{equation}
 (\alpha,\beta;i,j,k)\longleftrightarrow
 (\alpha,\beta;j,i,k).
\mlabel{eq:Tinvolution}
\end{equation}
\end{lem}

\begin{proof}
Equation~\meqref{eq:ncmn2} is precisely
Eq.~\meqref{eq:SsymNC}.  Moreover,
Eq.~\meqref{eq:ncmn1} can be rewritten as
\[
 x_i\rhd_\alpha(x_j\rhd_\beta x_k)
 -(x_i\rhd_\alpha x_j)\rhd_\beta x_k
 =x_j\rhd_\alpha(x_i\rhd_\beta x_k)
 -(x_j\rhd_\alpha x_i)\rhd_\beta x_k,
\]
which is Eq.~\meqref{eq:TsymNC}.To connect these algebraic identities with the corrected operadic relators in Eq.~\meqref{eq:ncmnovrel}, a direct use of Eqs.~\meqref{eq:comp1}-\meqref{eq:comp3} gives
\[
\begin{aligned}
 A_{\alpha,\beta}
 &=-\tform_{\alpha,\beta}(x,y;z)
   +\tform_{\alpha,\beta}(y,x;z),\\
 B_{\alpha,\beta}
 &=\sform_{\alpha,\beta}(x;y,z)
   -\sform_{\beta,\alpha}(x;z,y).
\end{aligned}
\]
Their $S_3$-translates therefore give exactly the displayed generators.

By Lemma~\mref{lem:STbasis}, the $\sform$-terms and $\tform$-terms form two
disjoint parts of a basis.  The generators in Eq.~\meqref{eq:SsymNC} only
identify pairs of $\sform$-basis vectors, while those in
Eq.~\meqref{eq:TsymNC} only identify pairs of $\tform$-basis vectors.  A
standard linear-algebra fact now applies: if a vector space has basis $B$
and one quotients by the span of all differences $b-b'$ belonging to a
specified equivalence relation on $B$, then the quotient has one basis vector
for each equivalence class.  Here the two equivalence relations are generated
by the involutions~\meqref{eq:Sinvolution} and~\meqref{eq:Tinvolution}.
\end{proof}

We now prove the kernel statement.  The proof is written out in some detail
because it is here that the freeness of Theorem~\mref{t:fncmdca} is used.
The equality established below is a statement inside $\calf(V)(3)$;
it is not a presentation of the entire operadic kernel $\ker\pi\morpre_{\mathrm{NC}}$.

\begin{thm}\mlabel{thm:omnovfromncmdca}
The operad $\oncmnov$ is the \precomp operad $\mpre(\oncmdca)$ of the operad
$\oncmdca$ of \ncmdcas defined in Eq.~\meqref{eq:oncmdca}.  More precisely,
for the operadic homomorphism
\[
 \morpre_{\mathrm{NC}}:\calf(V)\longrightarrow\calf(W),
 \qquad
\novop_\omega\mapsto\mu(\id\otimes\rho_\omega),\qquad
 \novop'_\omega\mapsto\mu(\rho_\omega\otimes\id),
\]
we have
\begin{equation}
 \ker\pi\morpre_{\mathrm{NC}}\cap\calf(V)(3)
 =\dfrel_{\oncmnov}.
\mlabel{eq:NCker}
\end{equation}
\end{thm}

\begin{proof}
We prove the two inclusions separately.

\smallskip
\noindent
\emph{Step 1: $\dfrel_{\oncmnov}\subseteq
\ker\pi\morpre_{\mathrm{NC}}\cap\calf(V)(3)$.}
Under the induced products
\[
 u\rhd_\omega v:=uD_\omega(v),
\]
we have, using the Leibniz rule,
\begin{align}
\morpre_{\mathrm{NC}}
 \big(\sform_{\alpha,\beta}(u;v,w)\big)
 &=uD_\alpha(v)D_\beta(w),
\mlabel{eq:SimageNC}\\
\morpre_{\mathrm{NC}}
 \big(\tform_{\alpha,\beta}(u,v;w)\big)
 &=uvD_\alpha D_\beta(w).
\mlabel{eq:TimageNC}
\end{align}
Indeed,
\begin{align*}
 u\rhd_\alpha(v\rhd_\beta w)
 &=uD_\alpha\big(vD_\beta(w)\big)\\
 &=uD_\alpha(v)D_\beta(w)+uvD_\alpha D_\beta(w),
\end{align*}
so subtracting $(u\rhd_\alpha v)\rhd_\beta w
=uD_\alpha(v)D_\beta(w)$ gives Eq.~\meqref{eq:TimageNC}.
Since the underlying multiplication is commutative,
\[
 uD_\alpha(v)D_\beta(w)
 =uD_\beta(w)D_\alpha(v)
\]
and
\[
 uvD_\alpha D_\beta(w)=vuD_\alpha D_\beta(w).
\]
Therefore the generators in Eqs.~\meqref{eq:SsymNC} and
\meqref{eq:TsymNC} map to zero, proving the first inclusion.  This is also
the content of the corresponding construction in~\mcite{WGZ1}.

\smallskip
\noindent
\emph{Step 2: the reverse inclusion.}
Let
\[
 r\in\ker\pi\morpre_{\mathrm{NC}}\cap\calf(V)(3).
\]
We must prove that $r\in\dfrel_{\oncmnov}$.  Let $\overline r$ be the class
of $r$ in
\[
 \calf(V)(3)/\dfrel_{\oncmnov}.
\]
By Lemma~\mref{lem:NCnormal}, choose one representative from each orbit in
Eqs.~\meqref{eq:Sinvolution} and~\meqref{eq:Tinvolution}.  Then
$\overline r$ has a unique expression
\begin{equation}
 \overline r
 =\sum_{p\in I_{\sform}} a_p\,\overline{\sform_p}
  +\sum_{q\in I_{\tform}} b_q\,\overline{\tform_q},
\mlabel{eq:rbarST}
\end{equation}
with only finitely many nonzero coefficients, where $I_{\sform}$ and
$I_{\tform}$ index the chosen orbit representatives.

Evaluate Eq.~\meqref{eq:rbarST} in the free noncommuting
multi-differential commutative algebra
\[
 A_{\mathrm{NC}}:=\bfk_{{\rm NC},\Omega}\{x_1,x_2,x_3\}
\]
of Theorem~\mref{t:fncmdca}.  Recall why this evaluation is sufficient.
This polynomial algebra is an $\oncmdca$-algebra. Since $r$ is in the operadic kernel, its value on $x_1,x_2,x_3$ must vanish. We show that the chosen normal-form images are independent in this one test algebra. Coupled with the already proved inclusion of the displayed relations in the operadic kernel, that suffices for the reverse inclusion; no all-degree embedding theorem is invoked. By the construction preceding that theorem,
$A_{\mathrm{NC}}$ is the polynomial algebra on the algebraically independent
symbols
\[
 (d_{\omega_1}\cdots d_{\omega_s})(x_i),
 \qquad s\geq 0,
\]
where the word $d_{\omega_1}\cdots d_{\omega_s}$ is retained with its
order.  Hence distinct commutative monomials in these symbols are linearly
independent.

For a chosen $\sform$-representative, the image is
\begin{equation}
 x_iD_\alpha(x_j)D_\beta(x_k).
\mlabel{eq:Sbasisimage}
\end{equation}
Two monomials of the form~\meqref{eq:Sbasisimage} are equal in
$A_{\mathrm{NC}}$ if and only if their undifferentiated variable is the same
and their two differentiated factors agree as an unordered pair.  Since
$i,j,k$ are distinct, this means precisely that the corresponding
$\sform$-basis elements are related by
Eq.~\meqref{eq:Sinvolution}.  Because we have already selected one
representative from each orbit, distinct chosen $\sform$-representatives have
distinct images.

For a chosen $\tform$-representative, the image is
\begin{equation}
 x_i x_j D_\alpha D_\beta(x_k).
\mlabel{eq:Tbasisimage}
\end{equation}
Here the ordered differential word $D_\alpha D_\beta$ is part of the
free differential variable.  Hence two monomials of the form
\meqref{eq:Tbasisimage} are equal if and only if they have the same $x_k$,
the same ordered pair $(\alpha,\beta)$, and the same unordered pair
$\{x_i,x_j\}$.  This is exactly the equivalence generated by
Eq.~\meqref{eq:Tinvolution}.  Thus distinct chosen $\tform$-representatives
also have distinct images.

Finally, an $\sform$-image cannot equal a $\tform$-image.  In
Eq.~\meqref{eq:Sbasisimage}, two distinct free differential variables carry
one derivation each, whereas in Eq.~\meqref{eq:Tbasisimage} a single free
differential variable carries the word of length two.  These are different
polynomial generators and hence give different monomials.

We have proved that the images of all basis vectors occurring in
Eq.~\meqref{eq:rbarST} are distinct basis monomials of $A_{\mathrm{NC}}$.
Since $r\in\ker\pi\morpre_{\mathrm{NC}}$, their linear combination is zero.
Linear independence forces
\[
 a_p=0\quad(p\in I_{\sform}),\qquad
 b_q=0\quad(q\in I_{\tform}).
\]
Thus $\overline r=0$, so
$r\in\dfrel_{\oncmnov}$.  This proves the reverse inclusion and hence
Eq.~\meqref{eq:NCker}.
\end{proof}

\section{The \precompst of commuting multi-differential commutative algebras}
\mlabel{sec:diffcomp}

\nc{\extrarel}{\mathcal C_\Omega}
\nc{\qfree}{q_{\Omega}}

We now pass from the noncommuting case to the commuting one.  By
Definition~\mref{df:operads}, the operad $\ocmdca$ is the quotient of
$\oncmdca$ obtained by imposing
\[
 \rho_\alpha\rho_\beta-\rho_\beta\rho_\alpha,
 \qquad \alpha,\beta\in\Omega.
\]
Let
\begin{equation}
 \quotcom:\oncmdca\longrightarrow\ocmdca
\mlabel{eq:commquot}
\end{equation}
be the quotient morphism.  The induced morphisms from the same free binary
operad satisfy
\begin{equation}
 \morpre=\quotcom\circ\morpre_{\mathrm{NC}}.
\mlabel{eq:phifactor}
\end{equation}
The purpose of the next lemma is to determine carefully what additional
arity-three kernel is created by this quotient.  This is the point at which
the extra multi-Novikov axiom appears.

\begin{defi}
\mlabel{cmna}
\mcite{BrD} A \textbf{\cmna} is a vector space $\text{N}$ equipped with
bilinear products $\rhd_\omega$ indexed by a set $\Omega$, satisfying
\begin{eqnarray}
&(x\rhd_\omega y)\rhd_\tau z-x\rhd_\omega (y\rhd_\tau z)
=(y\rhd_\omega x)\rhd_\tau z-y\rhd_\omega (x\rhd_\tau z),&
\mlabel{eq:mn1}\\
&(x\rhd_\omega y)\rhd_\tau z-x\rhd_\omega (y\rhd_\tau z)
=(x\rhd_\tau y)\rhd_\omega z-x\rhd_\tau(y\rhd_\omega z),&
\mlabel{eq:mn2}\\
&(x\rhd_\omega y)\rhd_\tau z=(x\rhd_\tau z)\rhd_\omega y,
\quad x,y,z\in\text{N},\ \omega,\tau\in\Omega.&
\mlabel{eq:mn3}
\end{eqnarray}
\end{defi}

For the operad encoding multi-Novikov algebras, let
$V=\bfk\{\novop_\omega,\novop_\omega'\mid\omega\in\Omega\}$ be as in
Definition~\mref{d:oncmnov} and define
\begin{equation}
\begin{split}
{\dfrel}_{\omnov}:=\operatorname{span}_{\bfk}
\big\{r\cdot\sigma\mid{}&\sigma\in S_3,\
 r\in\{A_{\alpha,\beta},B_{\alpha,\beta},C_{\alpha,\beta}\big\},
 \ \alpha,\beta\in\Omega\},\\
C_{\alpha,\beta}:={}&
 \novop_\beta\circ_\rmi\novop_\alpha
 -\novop'_\alpha\circ_\rmii\novop_\beta
 -\novop_\alpha\circ_\rmi\novop_\beta
 +\novop'_\beta\circ_\rmii\novop_\alpha.
\end{split}
\mlabel{eq:mnovrel}
\end{equation}
The quotient operad
\begin{equation}
 \omnov:=\calp(V,\dfrel_{\omnov})
\mlabel{d:omnov}
\end{equation}
is the operad of multi-Novikov algebras.

Before comparing kernels, it is helpful to make the quotient on the free
differential algebras explicit.  On the sets of differential variables there
is a natural map
\[
 M_{\rm NC}(d_\Omega)\times\{x_1,x_2,x_3\}
 \longrightarrow
 M(d_\Omega)\times\{x_1,x_2,x_3\}
\]
which sends a word $d_{\omega_1}\cdots d_{\omega_s}$ to the commutative
monomial with the same multiplicities of letters.  Extending multiplicatively
gives a surjective algebra homomorphism
\begin{equation}
 \qfree:
 \bfk_{{\rm NC},\Omega}\{x_1,x_2,x_3\}
 \longrightarrow
 \bfk_\Omega\{x_1,x_2,x_3\}.
\mlabel{eq:qfree}
\end{equation}
It is compatible with the quotient~\meqref{eq:commquot}: evaluating an
arity-three induced expression first in the noncommuting free algebra and
then applying $\qfree$ gives the same result as evaluating it in the free
commuting multi-differential algebra.  As in the proof of
Theorem~\mref{thm:omnovfromncmdca}, evaluation on the three free generators
detects the arity-three operadic kernel.  Consequently, determining which of
the normal-form basis monomials become equal under $\qfree$ is equivalent
to determining the additional arity-three kernel created by imposing
commutativity of the derivations.

\begin{lem}
\mlabel{lem:commquotient}
Let
$
 Q_{\mathrm{NC}}
 :=\calf(V)(3)/\dfrel_{\oncmnov}.
$
Via Theorem~\mref{thm:omnovfromncmdca}, $Q_{\mathrm{NC}}$ is identified with
the span in the free noncommuting differential algebra of the monomials in
Eqs.~\meqref{eq:Sbasisimage} and~\meqref{eq:Tbasisimage}, with one
representative for each orbit in
Eqs.~\meqref{eq:Sinvolution} and~\meqref{eq:Tinvolution}.

The kernel of the induced map
\begin{equation}
 Q_{\mathrm{NC}}
 \longrightarrow
 \bfk_\Omega\{x_1,x_2,x_3\}
\mlabel{eq:qonQ}
\end{equation}
is spanned by the classes of
\begin{equation}
 \tform_{\alpha,\beta}(x_i,x_j;x_k)
 -\tform_{\beta,\alpha}(x_i,x_j;x_k),
\mlabel{eq:Textra}
\end{equation}
where $(i,j,k)$ is a permutation of $(1,2,3)$ and
$\alpha,\beta\in\Omega$.  Equivalently, if
\begin{equation}
 \extrarel:=\bfk S_3\{\tform_{\alpha,\beta}(x,y;z)
 -\tform_{\beta,\alpha}(x,y;z)\mid \alpha,\beta\in\Omega\},
\mlabel{eq:defextra}
\end{equation}
then
\begin{equation}
 \ker\pi\morpre\cap\calf(V)(3)
 =\dfrel_{\oncmnov}+\extrarel.
\mlabel{eq:commkerintermediate}
\end{equation}
\end{lem}

\begin{proof}
We analyze separately the $\sform$- and $\tform$-basis classes of
$Q_{\mathrm{NC}}$ from Lemma~\mref{lem:NCnormal}.

For an $\sform$-class, the image under Eq.~\meqref{eq:qonQ} is
\[
 x_iD_\alpha(x_j)D_\beta(x_k).
\]
Only one derivation is applied to each differentiated variable.  Therefore
replacing the free noncommuting differential words by commutative ones does
not create any new equality inside this family.  More explicitly, two such
monomials in the free commuting differential algebra are equal if and only if
they have the same undifferentiated factor $x_i$ and the same two first-order
differential variables as an unordered pair.  This is precisely the
identification already imposed by Eq.~\meqref{eq:Sinvolution}.  Thus the
restriction of Eq.~\meqref{eq:qonQ} to the span of the chosen $\sform$-classes
is injective.

For a $\tform$-class, the image is
\begin{equation}
 x_i x_j D_\alpha D_\beta(x_k).
\mlabel{eq:TimageC}
\end{equation}
In the free commuting differential algebra, the differential variable
$D_\alpha D_\beta(x_k)$ depends only on the multiplicities of the operator
labels.  In length two this means
\[
 D_\alpha D_\beta(x_k)=D_\beta D_\alpha(x_k).
\]
No other new equality occurs.  Indeed, two monomials of the form
\meqref{eq:TimageC} are equal if and only if they have the same differentiated
base variable $x_k$, the same unordered pair $\{x_i,x_j\}$ of
undifferentiated variables, and the same unordered multiset
$\{\alpha,\beta\}$ of derivation labels.  The first two conditions are
already accounted for by Eq.~\meqref{eq:Tinvolution}; the last condition adds
exactly the transposition $(\alpha,\beta)\leftrightarrow(\beta,\alpha)$.
Hence the kernel on the $\tform$-part is spanned by the differences in
Eq.~\meqref{eq:Textra}.

As in the noncommuting case, an $\sform$-image cannot equal a $\tform$-image:
the former contains two first-order differential variables on two distinct
base variables, while the latter contains one second-order differential
variable on a single base variable.  Consequently the full kernel of
Eq.~\meqref{eq:qonQ} is exactly the span of
Eq.~\meqref{eq:Textra}.

Finally, the quotient map
$\calf(V)(3)\twoheadrightarrow Q_{\mathrm{NC}}$ has kernel
$\dfrel_{\oncmnov}$ by Theorem~\mref{thm:omnovfromncmdca}.  Taking the
preimage of the just-computed kernel under this quotient gives
Eq.~\meqref{eq:commkerintermediate}.
\end{proof}

We now identify the additional subspace $\extrarel$ with the extra
multi-Novikov axiom.

\begin{lem}
\mlabel{lem:extraMN}
With $\extrarel$ as in Eq.~\meqref{eq:defextra}, one has
\begin{equation}
 \dfrel_{\omnov}=\dfrel_{\oncmnov}+\extrarel.
\mlabel{eq:relquot}
\end{equation}
Equivalently, there is a canonical quotient of binary quadratic operads
\begin{equation}
 \oncmnov\twoheadrightarrow\omnov
\mlabel{eq:novquot}
\end{equation}
obtained by imposing the relations in Eq.~\meqref{eq:Textra}.
\end{lem}

\begin{proof}
Expanding Eq.~\meqref{eq:Tform}, the equality
\[
 \tform_{\alpha,\beta}(x,y;z)
 =\tform_{\beta,\alpha}(x,y;z)
\]
is
\begin{align*}
&x\rhd_\alpha(y\rhd_\beta z)
 -(x\rhd_\alpha y)\rhd_\beta z\\
&\qquad =x\rhd_\beta(y\rhd_\alpha z)
 -(x\rhd_\beta y)\rhd_\alpha z.
\end{align*}
Multiplying by $-1$ gives
\[
 (x\rhd_\alpha y)\rhd_\beta z
 -x\rhd_\alpha(y\rhd_\beta z)
 =
 (x\rhd_\beta y)\rhd_\alpha z
 -x\rhd_\beta(y\rhd_\alpha z),
\]
which is precisely Eq.~\meqref{eq:mn2}.
The third operadic relator has the corresponding form
\[
C_{\alpha,\beta}
=-\tform_{\alpha,\beta}(x,y;z)
 +\tform_{\beta,\alpha}(x,y;z).
\]
The other two multi-Novikov
relations, Eqs.~\meqref{eq:mn1} and~\meqref{eq:mn3}, are exactly the two
noncommuting multi-Novikov relations
Eqs.~\meqref{eq:ncmn1} and~\meqref{eq:ncmn2}.  Taking all $S_3$-translates
therefore gives Eq.~\meqref{eq:relquot}.
\end{proof}

We can now prove the commuting theorem with both inclusions visible.
Again, the conclusion is the exact \emph{arity-three} kernel, independently of any all-degree free-realization theorem.

\begin{thm}\mlabel{thm:omnovfromcmdca}
The operad $\omnov$ is the \precomp operad $\mpre(\ocmdca)$ of the operad
$\ocmdca$ of commuting multi-differential commutative algebras defined in
Eq.~\meqref{eq:ocmdca}.  More precisely, for
\[
 \morpre:\calf(V)\longrightarrow \calf(W),
 \qquad
\novop_\omega\mapsto\mu(\id\otimes\rho_\omega),\qquad
 \novop'_\omega\mapsto\mu(\rho_\omega\otimes\id),
\]
we have
\begin{equation}
 \ker\pi\morpre\cap\calf(V)(3)=\dfrel_{\omnov}.
\mlabel{eq:Cker}
\end{equation}
\end{thm}

\begin{proof}
Again we spell out the two inclusions.

First, Eqs.~\meqref{eq:mn1} and~\meqref{eq:mn3} already hold in the
noncommuting case by Step~1 of the proof of
Theorem~\mref{thm:omnovfromncmdca}.  For Eq.~\meqref{eq:mn2}, the difference
of its two sides is, up to sign,
\[
 \tform_{\alpha,\beta}(x,y;z)
 -\tform_{\beta,\alpha}(x,y;z),
\]
whose image is
\[
 xy\big(D_\alpha D_\beta-D_\beta D_\alpha\big)(z)=0
\]
in every commuting multi-differential commutative algebra.  Hence
\[
 \dfrel_{\omnov}
 \subseteq\ker\pi\morpre\cap\calf(V)(3).
\]

Conversely, let
$r\in\ker\pi\morpre\cap\calf(V)(3)$.  By
Lemma~\mref{lem:commquotient}, the complete arity-three kernel of the
commuting map is
\[
 \dfrel_{\oncmnov}+\extrarel.
\]
By Lemma~\mref{lem:extraMN}, this sum is exactly
$\dfrel_{\omnov}$.  Therefore
$r\in\dfrel_{\omnov}$, proving the reverse inclusion.  This proves
Eq.~\meqref{eq:Cker}.
\end{proof}

Specializing to the one-operator case, we obtain the following.
\begin{coro}
\mlabel{co:novind}
The operad of Novikov algebras is the \precomp operad of the operad of
differential commutative algebras.
\end{coro}

\begin{remark}
\mlabel{r:novind}
Theorems~\mref{thm:omnovfromncmdca} and~\mref{thm:omnovfromcmdca} also explain
structurally the difference between the noncommuting and commuting cases.
After passing to the $\sform$--$\tform$ basis, the noncommuting kernel is
caused only by commutativity of the underlying multiplication: the two
first-order differentiated factors in an $\sform$-term may be exchanged, and
the two undifferentiated factors in a $\tform$-term may be exchanged.  When
the derivations themselves are required to commute, no new relation is
created on the $\sform$-part, while the $\tform$-part acquires the additional
symmetry
\[
 \tform_{\alpha,\beta}=\tform_{\beta,\alpha}.
\]
This is exactly the second multi-Novikov identity.  Thus the quotient
$\oncmdca\twoheadrightarrow\ocmdca$ is mirrored in arity three by the
quotient
$\oncmnov\twoheadrightarrow\omnov$ in Eq.~\meqref{eq:novquot}.

For a finite operator set of cardinality $m$, the derivative-distribution
basis also gives the arity-three image dimensions
\[
\dim\operatorname{im}(\morpre_{\mathrm{NC}}|_{\calf(V)(3)})=6m^2,
\qquad
\dim\operatorname{im}(\morpre|_{\calf(V)(3)})
 =3m^2+3\binom{m+1}{2}.
\]
Thus requiring the derivations to commute adds precisely
$3\binom{m}{2}$ independent quadratic relations.
More precisely, this is the dimension of the quotient relation space
\[
\dim(\dfrel_{\omnov}/\dfrel_{\oncmnov})=3\binom m2.
\]
It differs from the raw extra-relation space: $\dim\extrarel=6\binom m2$ and $\dim(\extrarel\cap\dfrel_{\oncmnov})=3\binom m2$.
Indeed, there are
$3m^2$ classes of type $\sform$ in either case, and respectively
$3m^2$ and $3\binom{m+1}{2}$ classes of type $\tform$.

In particular, Theorem~\mref{thm:omnovfromcmdca} states that all binary
quadratic relations induced from commuting multi-differential commutative
algebras are generated by the multi-Novikov relations.  In the one-operator
case, Corollary~\mref{co:novind} recovers the corresponding Novikov statement.
\end{remark}

\noindent
{\bf Acknowledgements.}
Xiaoyan Wang was  supported by  the National Key R$\&$D Program of China (No. 2024YFA1013803)  and by Shanghai Key Laboratory of PMMP (No. 22DZ2229014).
Huhu Zhang is supported by 
the National Natural Science Foundation of China (12601044, 12671025),
the Natural Science Basic Research Program of Shaanxi (2026JC-YBQN-0018) and
the Young Talent Fund of Association for Science and Technology in Shaanxi, China (20250530).

\noindent
{\bf Declaration of interests. } The authors have no conflicts of interest to disclose.

\noindent
{\bf Data availability. } Data sharing is not applicable as no data were created or analyzed.

\end{document}